\documentclass{article}
\usepackage[utf8]{inputenc}
\usepackage{amssymb}
\usepackage{amsmath}
\usepackage{amsthm}
\usepackage{tikz}
\usepackage{url}
\usepackage{xcolor}
\usepackage{mathtools}
\usepackage{geometry}
\usepackage{enumerate}
\usepackage{accents}
\usepackage{hyperref}
\usepackage{xkeyval}
\usepackage[affil-it]{authblk}
\usepackage{stmaryrd}
\usepackage{enumitem}

\usetikzlibrary{calc}

\newcommand{\scrI}{\mathcal{I}}

\newcommand{\scrJ}{\mathcal{J}}

\renewcommand{\P}{\mathbb{P}}

\newcommand{\Q}{\mathbb{Q}}
\newcommand{\R}{\mathbb{R}}

\newcommand{\append}{{}^\frown}
\newcommand{\boldsig}{\boldsymbol{\Sigma}}

\newcommand\forces{\Vdash}

\newcommand{\frakb}{\mathfrak{b}}
\newcommand{\frakd}{\mathfrak{d}}
\newcommand{\frakc}{\mathfrak{c}}

\newcommand{\Pow}{\mathcal{P}}
\newcommand{\non}{\operatorname{non}}
\newcommand{\cov}{\operatorname{cov}}
\newcommand{\add}{\operatorname{add}}
\newcommand{\cof}{\operatorname{cof}}

\newcommand{\nul}{\mathcal{N}}
\newcommand{\meager}{\mathcal{M}}

\newcommand{\cf}{\operatorname{cf}}
\newcommand{\Add}{\operatorname{Add}}

\newcommand{\suc}{\operatorname{succ}}
\newcommand{\Lev}{\operatorname{Lev}}
\newcommand{\restrict}{\upharpoonright}

\newcommand{\ZFC}{\mathsf{ZFC}}
\newcommand{\CH}{\mathsf{CH}}

\newcommand{\Miller}{\mathbb{M}}

\newcommand{\MeC}{\operatorname{Meas}(2^\omega)}
\newcommand{\PrC}{\operatorname{Prob}(2^\omega)}

\newcommand{\PTfg}{\mathbb{PT}_{f,g}}
\newcommand{\UN}{\mathcal{UN}}

\newcommand{\SN}{\mathcal{SN}}
\newcommand{\Lb}{\mathsf{Lb}}
\newcommand{\baA}{\mathbb{A}}
\newcommand{\Rpt}{\mathbf{R}_\mathrm{pt}}
\newcommand{\truth}[1] {\llbracket #1 \rrbracket}

\newcommand{\splt}{\operatorname{sp}}

\newcommand{\seq}[1]{{\langle#1\rangle}}
\DeclarePairedDelimiter\abs{\lvert}{\rvert}

\renewcommand\emptyset{\varnothing}
\renewcommand\subset{\subseteq}
\renewcommand{\setminus}{\smallsetminus}

\renewcommand{\le}{\leqslant}
\renewcommand{\ge}{\geqslant}
\renewcommand{\leq}{\leqslant}
\renewcommand{\geq}{\geqslant}

\newcommand{\subjclass}[2][2020]{%
  \let\oldthefootnote\thefootnote%
  \renewcommand{\thefootnote}{}%
  \footnotetext{\emph{2020 Mathematics Subject Classification.} #2}%
  \let\thefootnote\oldthefootnote%
}
\newcommand{\keywords}[1]{%
  \let\oldthefootnote\thefootnote%
  \renewcommand{\thefootnote}{}%
  \footnotetext{\emph{Key words and phrases.} #1}%
  \let\thefootnote\oldthefootnote%
}

\theoremstyle{definition}
\newtheorem{thm}{Theorem}[section]
\newtheorem*{thm*}{Theorem}
\newtheorem{defi}[thm]{Definition}
\newtheorem*{defi*}{Definition}
\newtheorem{lem}[thm]{Lemma}
\newtheorem{prop}[thm]{Proposition}
\newtheorem*{lem*}{Lemma}
\newtheorem{fact}[thm]{Fact}
\newtheorem*{fact*}{Fact}
\newtheorem{rmk}[thm]{Remark}
\newtheorem*{rmk*}{Remark}
\newtheorem{cor}[thm]{Corollary}
\newtheorem*{cor*}{Corollary}
\newtheorem*{convention*}{Convention}
\newtheorem*{notation*}{Notation}
\newtheorem{question}[thm]{Question}

\newtheorem{conjecture}[thm]{Conjecture}

\usepackage[backend=biber,style=alphabetic,sorting=nty,doi=false,isbn=false,url=false,eprint=false]{biblatex}
\renewbibmacro{in:}{}

\definecolor{myred}{RGB}{230, 150, 100}
\definecolor{mygreen}{RGB}{150, 230, 100}
\definecolor{myblue}{RGB}{130, 220, 220}
\hypersetup{linkbordercolor=myred,citebordercolor=mygreen,urlbordercolor=myblue}

\usepackage{xpatch}
\newcounter{proofdepth}

\xpretocmd{\proof}{%
	\stepcounter{proofdepth}%
	\ifnum\value{proofdepth}=1
	\else
	\fi
}{}{}

\xapptocmd{\endproof}{%
	\addtocounter{proofdepth}{-1}%
}{}{}
\newcounter{claim}[thm]

\newtheorem{innerclaim}{Claim}

\newenvironment{claim}{%
	\refstepcounter{claim}%
	\begin{innerclaim}
	}{%
	\end{innerclaim}
}

\title{Forcing and cardinal invariants of universally null sets}
\author{Tatsuya Goto}

\affil{Institute of Discrete Mathematics and Geometry, TU Wien, Wiedner Hauptstra\ss{}e 8-10/104, 1040 Wien, Austria.
\\e-mail: goto.tatsuya@icloud.com}
\date{\today}

\begin{document}
	\maketitle
    

    \begin{abstract}
        We study universally null sets in forcing extensions and the
        cardinal invariants of their ideal $\mathcal{UN}$.  We prove that the
        Miller model contains a nonmeager universally null set of
        cardinality $\aleph_2=\mathfrak{c}$, answering a question of Brendle
        and Larson.  The proof uses preservation of universal nullity
        by countable support iterations of Miller forcing and shows
        that the set of Miller reals added during the iteration is
        universally null.  We also prove that every universally null
        set in the $\mathbb{PT}_{f,g}$ model has cardinality at most $\aleph_1$,
        and show that there is consistently a forcing destructible universally null set.  For the cardinal invariants, we show that
        $|X|<\operatorname{cof}(\mathcal{UN})$ for every
        $X\in\mathcal{UN}$, and that
        $\operatorname{add}(\mathcal{N})=\mathfrak{c}$ implies
        $\operatorname{cof}(\mathcal{UN})=\mathfrak{d}_{\mathfrak{c}}$.
        Finally, we prove the
        consistency of
        $\operatorname{add}(\mathcal{UN})<\operatorname{cov}(\mathcal{UN})
        <\operatorname{non}(\mathcal{UN})<\operatorname{cof}(\mathcal{UN})$.
    \end{abstract}

    \section{Introduction}

    A subset of $2^\omega$ or $\omega^\omega$ is \emph{universally
    null} if it has outer measure zero for every atomless finite
    Borel measure on that space.  Universally null sets form a
    $\sigma$-ideal, denoted by $\UN$, and contain no nonempty perfect
    sets.  On $2^\omega$, we have
    \[
        \SN\subseteq\UN\subseteq\nul,
    \]
    where $\SN$ is the strong measure zero ideal and $\nul$ is the
    null ideal for the usual product measure.  We study the behavior
    of universally null sets under forcing and the additivity,
    covering number, uniformity, and cofinality of $\UN$.

    The possible sizes of universally null sets provide one
    starting point.  Grzegorek \cite{grzegorek1980solution}
    constructed such a set of cardinality $\non(\nul)$, and
    Rec\l aw \cite{reclaw} showed that every set well-ordered by a
    Borel relation is universally null.  On the other hand, in the
    iterated Sacks model every universally null set has cardinality
    at most $\aleph_1$
    \cite[Theorem~1.1.4]{Ciesielski_Pawlikowski_2004}.
    Thus the existence of a universally null set of size continuum
    depends on the model.

    In their unpublished note \emph{Nonmeager universally null
    sets}, Brendle and Larson \cite{BrendleLarson2017Nonmeager}
    prove that
    \[
        \min\{\frakd,\non(\nul)\}=\cof(\meager)
    \]
    implies the existence of a nonmeager universally null set of
    cardinality $\cof(\meager)$.  Here $\meager$ denotes the meager
    ideal, $\frakd$ is the dominating number, and
    $\frakc=2^{\aleph_0}$.  Their hypothesis fails in the Miller
    model, where
    \[
        \non(\nul)=\aleph_1<\frakd=\cof(\meager)=\frakc=\aleph_2.
    \]
    Thus their theorem does not settle the existence of a
    nonmeager universally null set of size continuum in this
    model.  They ask, in particular, whether the
    Miller model contains a nonmeager universally null set of
    cardinality $\aleph_2$.  We answer this question affirmatively.

    \begin{thm*}[Corollary~\ref{cor:brendle_larson}]
        In the Miller model there exists a nonmeager universally
        null set of cardinality $\aleph_2=\frakc$.
    \end{thm*}

    We use the standard Miller model obtained by a countable
    support iteration of length $\omega_2$ over a model of $\CH$.
    The main preservation result is that countable support
    iterations of Miller forcing of length at most $\omega_2$
    preserve ground model universally null sets
    (Corollary~\ref{cor:preservation}).  We then show that the set
    of Miller reals added at the individual stages is universally
    null (Theorem~\ref{thm:miller_reals_universally_null}).
    Adjoining a preserved nonmeager universally null set from the
    ground model gives the required example.

    We also consider two other aspects of universal nullity in
    forcing extensions.  In the model obtained by a countable
    support iteration of $\PTfg$, every universally null set is
    contained in the reals of some intermediate model and hence
    has cardinality at most $\aleph_1$
    (Theorem~\ref{thm:ptfg_no_large_un}).  The proof uses the
    results on polar forcings and measured extensions of Larson
    and Zapletal \cite{LarsonZapletal2023Polar}.  This model also
    lies outside the scope of the Brendle--Larson theorem: it
    satisfies
    \[
        \frakd=\aleph_1<\non(\meager)=\cof(\meager)=\frakc=\aleph_2,
    \]
    so again $\min\{\frakd,\non(\nul)\}<\cof(\meager)$.
    These values follow from the bounding and category properties
    of $\PTfg$ and its countable support iterations
    \cite[Section~2]{BartoszynskiJudahShelah1993Cichon}.

    Furthermore,
    universal nullity need not be preserved even by an
    $\omega_1$-preserving forcing: using a construction of
    Veli\v{c}kovi\'{c} and Woodin
    \cite{VelickovicWoodin1998Complexity}, we obtain a model with
    a universally null Luzin set that acquires a nonempty perfect
    subset in such an extension (Theorem~\ref{thm:destructible}).
    This answers affirmatively the existence part of a question
    left open in Larson's lecture slides of February~3, 2026
    \cite[p.~24]{Larson2026UniversallyMeasurableII}: whether a
    universally null set which is not universally meager can be
    made non-universally-null by forcing.

    For the cardinal invariants of $\UN$, our arguments draw on
    work on the strong measure zero ideal, particularly Yorioka's
    analysis of its cofinality \cite{yorioka2002cofinality} and
    the preservation methods of Brendle, Cardona, and Mej\'{\i}a
    \cite{brendle2025separating}.  We first establish a bound
    relating the size of an individual universally null set to
    the cofinality of the ideal.

    \begin{thm*}[Theorem~\ref{thm:lower_bound_of_cofinality}]
        For every $X\in\UN$, we have $|X|<\cof(\UN)$.
    \end{thm*}

    This yields
    \[
        \frakc\leq\cof(\UN),\qquad
        \frakb<\cof(\UN),\qquad
        \non(\nul)<\cof(\UN).
    \]
    In particular, the existence of a universally null set of
    size continuum in the Miller model implies
    $\frakc<\cof(\UN)$ there.  Adapting Yorioka's argument, we
    also prove that $\add(\nul)=\frakc$ implies
    $\cof(\UN)=\frakd_{\frakc}$, where $\frakd_\kappa$ is the
    dominating number of $\kappa^\kappa$ under pointwise
    domination.  Finally, a refinement of the preservation
    argument of Brendle, Cardona, and Mej\'{\i}a gives the
    consistency of
    \[
        \add(\UN)<\cov(\UN)<\non(\UN)<\cof(\UN).
    \]

    Sections~\ref{sec:miller}--\ref{sec:destructible} contain
    the forcing results.  Section~\ref{sec:cardinal_invariants}
    treats the cardinal invariants, and
    Section~\ref{sec:discussion} lists some remaining questions.

    \section{Universally null sets of size continuum in the Miller model}\label{sec:miller}

We use the standard Miller model obtained by a countable support
iteration of Miller forcing of length $\omega_2$ over a ground
model satisfying $\CH$.  We prove that
\[
X=\{m_\alpha:\alpha<\omega_2\}
\]
is universally null, where $m_\alpha$ is the Miller real added at
stage $\alpha$ (Theorem~\ref{thm:miller_reals_universally_null}).
Adjoining a nonmeager universally null set from the ground model
yields the stated theorem (Corollary~\ref{cor:brendle_larson}).

The main step is preservation of universally null sets
(Corollary~\ref{cor:preservation}).  We decompose each new atomless
measure into a part absolutely continuous with respect to an old
atomless measure and a part singular to all old atomless measures.
Using Zdomskyy's approximation lemma
\cite{Zdomskyy2018MengerMiller}, we show that the singular part
makes all old reals null.  Universal nullity of $X$ then follows
by transfinite induction.

We write $\MeC$ for the set of atomless finite Borel measures on
$2^\omega$, and $\PrC$ for the space of all Borel probability
measures on $2^\omega$, equipped with the weak topology.
With this topology, $\PrC$ is a compact Polish space.
We use $\mu^*$ for outer measure.
For finite Borel measures $\mu,\nu$ on $2^\omega$, write
$\mu\ll\nu$ (\emph{absolute continuity}) if every $\nu$-null Borel
set is $\mu$-null, and $\mu\perp\nu$ (\emph{mutual singularity})
if some Borel set $B$ satisfies $\mu(2^\omega\setminus B)=\nu(B)=0$.

Ground model Borel sets, maps, and measures are reinterpreted in
extensions via their codes.
Absolute continuity and mutual singularity of ground model
measures are absolute between these models.

    \subsection{Lemmas}

    \begin{lem}\label{lem:decomposition}
        Let $W$ be a proper forcing extension of $V$.
        Every $\mu\in\MeC^W$ admits a decomposition
        \begin{equation}\label{eq:measure_decomposition}
            \mu=\mu_{\mathrm{ac}}+\mu_{\mathrm{s}},
        \end{equation}
        into finite Borel measures such that, for some $\nu\in\MeC^V$,
        \begin{equation}\label{eq:decomposition_properties}
            \mu_{\mathrm{ac}}\ll\nu,
            \qquad
            \mu_{\mathrm{s}}\perp\lambda
            \quad\text{for every }\lambda\in\MeC^V.
        \end{equation}
    \end{lem}
    \begin{proof}
        In $V$, fix a maximal family $\mathcal A$ of pairwise mutually
        singular atomless probability measures.  In $W$, let
        \[
            S=\{\lambda\in\mathcal A:\mu\not\perp\lambda\}.
        \]
        For each $\lambda\in S$, let $f_\lambda$ be the
        Radon--Nikodym derivative of the absolutely continuous part
        of $\lambda$ with respect to $\mu$.
        The sets $\{ x \in 2^\omega :  f_\lambda(x) >0\}$ have positive $\mu$-measure and
        are pairwise disjoint modulo $\mu$-null sets, since the
        measures in $\mathcal A$ are mutually singular.
        Thus $S$ is countable.

        By the countable covering property of proper forcing, there is a
        set $B\in V$, countable in $V$, such that
        $S\subseteq B\subseteq\mathcal A$.  We may assume $B\ne\emptyset$.
        In $V$, choose a convex combination $\nu$ of the members of
        $B$ with all coefficients positive.  In $W$, take the Lebesgue
        decomposition $\mu=\mu_{\mathrm{ac}}+\mu_{\mathrm{s}}$ with
        respect to $\nu$.  Then
        $\mu_{\mathrm{ac}}\ll\nu$ and $\mu_{\mathrm{s}}\perp\nu$.

        For $\lambda\in B$, we have $\lambda\ll\nu$, whereas for
        $\lambda\in\mathcal A\setminus B$ we have $\mu\perp\lambda$.
        Hence $\mu_{\mathrm{s}}\perp\lambda$ for every
        $\lambda\in\mathcal A$.

        Fix $\eta\in\MeC^V$.  Working in $V$, the same argument shows
        that
        \[
            \mathcal A_\eta
            =\{\lambda\in\mathcal A:\eta\not\perp\lambda\}
        \]
        is countable.  Choose positive coefficients $c_\lambda$
        with $\sum_{\lambda\in\mathcal A_\eta}c_\lambda<\infty$, and put
        $\rho=\sum_{\lambda\in\mathcal A_\eta}c_\lambda\lambda$
        (with $\rho=0$ if $\mathcal A_\eta=\emptyset$).
        Take the Lebesgue decomposition
        $\eta=\eta_{\mathrm{ac}}+\eta_{\mathrm{s}}$ with respect to $\rho$.

        As above, $\eta_{\mathrm{s}}$ is singular to every member of
        $\mathcal A$.  It is atomless, so maximality of $\mathcal A$
        implies $\eta_{\mathrm{s}}=0$.  Thus $\eta\ll\rho$, also
        after reinterpretation in $W$.

        In $W$, the measure $\mu_{\mathrm{s}}$ is singular to each
        summand of the countable sum defining $\rho$, hence to $\rho$.
        Since $\eta\ll\rho$, we obtain $\mu_{\mathrm{s}}\perp\eta$.
    \end{proof}

    \begin{lem}\label{lem:conseq_of_adding_no_eventually_different_reals}
        Let $W\supseteq V$ add no real eventually different from every
        member of $(\omega^\omega)^V$.
        Let $\nu\in V$ and $\mu\in W$ be finite Borel measures on
        $2^\omega$, and suppose $\nu\perp\mu$.
        Then there is a $G_\delta$ set $D$ coded in $V$ such that
        $\nu(D)=0$ and $\mu(D)=\mu(2^\omega)$.
    \end{lem}
    \begin{proof}
        If $\mu=0$, the conclusion is immediate.  Otherwise, by
        normalizing $\mu$, we may assume that $\mu(2^\omega)=1$.
        Fix in $V$ an enumeration $\seq{A_i:i\in\omega}$ of the
        clopen subsets of $2^\omega$.
        By regularity and singularity, choose
        $r\in(\omega^\omega)^W$ such that, for every $j\in\omega$,
        \[
            \mu(A_{r(j)})>1-2^{-j},
            \qquad \nu(A_{r(j)})<2^{-j}.
        \]
        Choose $f\in(\omega^\omega)^V$ with $f(j)=r(j)$ for
        infinitely many $j$.  In $V$, define
        \[
            B_j=\begin{cases}
                A_{f(j)} & \text{if }\nu(A_{f(j)})<2^{-j},\\
                \emptyset & \text{otherwise},
            \end{cases}
            \qquad D=\bigcap_k\bigcup_{j\geq k}B_j.
        \]
        Since $\nu(\bigcup_{j\geq k}B_j)
        \leq\sum_{j\geq k}2^{-j}\to0$, we have $\nu(D)=0$.
        For each $k$, arbitrarily large $j\geq k$ satisfy
        $B_j=A_{r(j)}$, so $\mu(\bigcup_{j\geq k}B_j)=1$.
        Continuity from above gives $\mu(D)=1$.
    \end{proof}

    Miller forcing $\Miller$ consists of nonempty trees
    $T\subseteq\omega^{<\omega}$ of strictly increasing sequences,
    such that every node has a splitting extension and every splitting
    node has infinitely many immediate successors.  The order is inclusion.
    Write $\splt(T)$ for the splitting nodes of $T$ and
    $\suc_T(s)$ for the set of immediate successors of $s$ in $T$, and put
    \[
        \Lev(s,T)=|\{u\in\splt(T):u\subsetneq s\}|,
        \qquad T_t=\{u\in T:u\subseteq t\text{ or }t\subseteq u\}.
    \]
    Let $\Miller_\delta$ denote the countable support iteration of
    Miller forcing of length $\delta$.  For $p\in\Miller_\delta$
    and $t\in p(0)$, set
    $p_t=(p(0))_t\append p\restrict[1,\delta)$.

    We use the following consequence of
    \cite[Lemma 2.3]{Zdomskyy2018MengerMiller} for compact metric spaces.

    \begin{lem}\label{lem:zdomskyyslemma}
        Let $0<\delta\leq\omega_2$, let $K$ be a nonempty compact
        metric space coded in $V$, and let $\dot x$ be a
        $\Miller_\delta$-name for a point of $K$.
        For every $p\in\Miller_\delta$, there are $q\leq p$ and
        a family $\seq{U_s:s\in\splt(q(0))}$ of finite nonempty
        subsets of $K^V$ such that for every $s\in\splt(q(0))$
        and every open set $O\supseteq U_s$
        coded in $V$, all but finitely many immediate successors $t$
        of $s$ in $q(0)$ satisfy
        \[
            q_t\forces\dot x\in O.
        \]
    \end{lem}
    \begin{proof}
        First let $\delta=\omega_2$.  In $V$, choose a continuous
        surjection $F:C\to K$, where $C\subseteq\R$ is the Cantor
        set, and choose a name $\dot z\in C$ with $F(\dot z)=\dot x$.
        Apply \cite[Lemma 2.3]{Zdomskyy2018MengerMiller} to the
        constant sequence $\seq{\dot z:i\in\omega}$, obtaining
        $q\leq p$ and finite sets $A_s\subseteq\R$ such that
        \[
            \forall\varepsilon>0\ \forall^\infty t\in\suc_{q(0)}(s)\quad
            q_t\forces\exists a\in A_s\ (|\dot z-a|<\varepsilon).
        \]
        Here $\forall^\infty$ means ``for all but finitely many.''
        Since $C$ is closed and $A_s$ is finite,
        $A_s\cap C$ is nonempty and has the same approximation
        property.  Set $U_s=F[A_s\cap C]$ and apply continuity of $F$.

        For $\delta<\omega_2$, apply this result in the longer
        iteration and restrict $q$ to $\delta$, using completeness of
        the initial subiteration and absoluteness of membership in
        coded open sets.
    \end{proof}

    \begin{lem}\label{lem:singularity_implies_nullity_of_ground_model}
        Let $0<\delta\leq\omega_2$, $p\in\Miller_\delta$, and
        let $\dot\mu$ be a $\Miller_\delta$-name for a member of
        $\MeC$.  If $p$ forces that $\dot\mu$ is singular to every
        member of $\MeC^V$, then
        \[
            p\forces\dot\mu^*((2^\omega)^V)=0.
        \]
    \end{lem}
    \begin{proof}
        By density, it suffices to find a condition below $p$ forcing the
        conclusion.  Strengthening $p$ and normalizing $\dot\mu$, we
        may assume $p\forces\dot\mu(2^\omega)=1$, since the zero
        measure is trivial.  An atomless measure is singular to every
        ground model purely atomic measure.  Decomposing each ground
        model probability measure into its atomless and atomic parts
        therefore gives
        \[
            p\forces\dot\mu\perp\nu
            \qquad\text{for every }\nu\in\PrC^V.
        \]

        Apply Lemma~\ref{lem:zdomskyyslemma} to the compact space
        $\PrC$ and the name $\dot\mu$.  We obtain $q\leq p$ and
        finite nonempty sets $U_s\subseteq\PrC^V$, for
        $s\in\splt(q(0))$, such that for every open set $O\supseteq U_s$
        coded in $V$, all but finitely many immediate successors $t$
        of $s$ satisfy
        \begin{equation}\label{eq:measure_approximation}
            q_t\forces\dot\mu\in O.
        \end{equation}

        In $V$, choose a convex combination $\nu$ of the members of
        $\bigcup_{s\in\splt(q(0))}U_s$ with all coefficients positive.
        Miller iterations preserve nonmeager sets
        \cite[Theorem 10(3)]{Bartoszynski2002PerfectlyMeager}, and hence
        add no eventually different real: such a real would place
        $(\omega^\omega)^V$ in the meager set of reals eventually
        different from it.
        By Lemma~\ref{lem:conseq_of_adding_no_eventually_different_reals},
        we can choose $r\leq q$ and a $G_\delta$ set $D$ coded in $V$
        such that
        \begin{equation}\label{eq:ground_model_full_measure_set}
            \nu(D)=0,
            \qquad r\forces\dot\mu(D)=1.
        \end{equation}
        Put $T=r(0)$.  For $s\in\splt(T)$, we have
        $\lambda(D)=0$ for every $\lambda\in U_s$.
        Moreover, \eqref{eq:measure_approximation} holds with $r$ in
        place of $q$, since $\splt(T)\subseteq\splt(q(0))$ and
        $r_t\leq q_t$.

        For each $s\in\splt(T)$, put
        $\varepsilon_s=2^{-\Lev(s,T)-4}$.
        By outer regularity and finiteness of $U_s$, choose in $V$ an
        open set $O_s\supseteq D$ such that
        $\lambda(O_s)<\varepsilon_s$ for every $\lambda\in U_s$.
        Choose an increasing sequence of clopen sets
        $\seq{K_{s,k}:k\in\omega}$ with $K_{s,0}=\emptyset$ and
        $\bigcup_k K_{s,k}=O_s$.
        For each fixed $k$, the set
        \[
            W_{s,k}=\{\eta\in\PrC:
                \eta(K_{s,k})<2\varepsilon_s\}
        \]
        is open and contains $U_s$, since evaluation on clopen sets is
        continuous.  Thus all but finitely many immediate successors
        $t$ of $s$ in $T$ satisfy
        $r_t\forces\dot\mu(K_{s,k})<2\varepsilon_s$.

        Enumerate all immediate successors of $s$ in $T$ without
        repetition as $\seq{t_i^s:i\in\omega}$, and define in $V$
        \[
            j_s(t_i^s)=\max\{k\leq i:
                r_{t_i^s}\forces\dot\mu(K_{s,k})<2\varepsilon_s\}.
        \]
        The value $k=0$ is always admissible, and each fixed $k$ is
        admissible for all sufficiently large $i$.  Hence
        $j_s(t_i^s)\to\infty$, and
        \begin{equation}\label{eq:selected_clopen_bound}
            r_t\forces
                \dot\mu(K_{s,j_s(t)})<2\varepsilon_s.
        \end{equation}
        Here $t$ ranges over all immediate successors of $s$ in $T$.

        Let $\dot m_0$ be the first Miller real.  In a generic extension
        containing $r$, enumerate the splitting nodes of $T$ on $m_0$
        in increasing order as $s_0,s_1,\ldots$, and put
        $t_n=m_0\restrict(|s_n|+1)$.
        Let $\dot E$ name the Borel set
        \[
            E_{m_0}=D\cap\bigcup_{N\in\omega}\bigcap_{n\geq N}
                \bigl(2^\omega\setminus K_{s_n,j_{s_n}(t_n)}\bigr).
        \]
        Its code belongs to $V[m_0]$.  Since $r_{t_n}$ belongs to the
        generic filter,
        \eqref{eq:selected_clopen_bound} and $\Lev(s_n,T)=n$ give
        \[
            \sum_{n\in\omega}\mu(K_{s_n,j_{s_n}(t_n)})
            \leq\sum_{n\in\omega}2\varepsilon_{s_n}
            =\sum_{n\in\omega}2^{-n-3}<\infty.
        \]
        The Borel--Cantelli lemma and
        \eqref{eq:ground_model_full_measure_set} yield
        $r\forces\dot\mu(\dot E)=1$.

        Fix $x\in D\cap V$ and define in $V$
        \[
            h_x(s)=\min\{k:x\in K_{s,k}\}
            \qquad(s\in\splt(T)).
        \]
        This is defined since $x\in D\subseteq O_s$.  At each $s$,
        only finitely many successors $t$ satisfy $j_s(t)<h_x(s)$.
        Given any $R\leq T$, delete these successors and their cones
        at each splitting node of $R$.  The remaining tree $R'\leq R$
        is a Miller tree, and every branch through $R'$ satisfies
        \[
            x\in K_{s_n,j_{s_n}(t_n)}
            \qquad\text{for infinitely many }n.
        \]
        Thus, by density in the first coordinate,
        $r\forces x\notin\dot E$.  Points outside $D$ are excluded
        by definition, so
        \[
            r\forces\dot E\cap(2^\omega)^V=\emptyset.
        \]
        Hence $2^\omega\setminus\dot E$ is a Borel null cover of
        $(2^\omega)^V$ for the measure $\dot{\mu}$.
    \end{proof}

    \subsection{Conclusion}

    \begin{cor}\label{cor:preservation}
        Countable support iterations of Miller forcing of length at
        most $\omega_2$ preserve ground model universally null sets.
    \end{cor}
    \begin{proof}
        The iteration of length zero is trivial.  Let $W=V[G]$ be
        an extension by a nontrivial iteration, let
        $Y\subseteq2^\omega$ be universally null in $V$, and fix
        $\mu\in\MeC^W$.  By properness and
        Lemma~\ref{lem:decomposition}, write
        \[
            \mu=\mu_{\mathrm{ac}}+\mu_{\mathrm{s}},
            \qquad \mu_{\mathrm{ac}}\ll\nu
            \quad\text{for some }\nu\in\MeC^V,
        \]
        where $\mu_{\mathrm{s}}$ is singular to every member of
        $\MeC^V$.  In $V$, choose a Borel set $B\supseteq Y$ with
        $\nu(B)=0$.  In $W$, we have $\mu_{\mathrm{ac}}(B)=0$.

        Since $\mu_{\mathrm{s}}\leq\mu$ is atomless,
        Lemma~\ref{lem:singularity_implies_nullity_of_ground_model}
        yields a Borel set $C$ in $W$ such that
        \[
            (2^\omega)^V\subseteq C,
            \qquad \mu_{\mathrm{s}}(C)=0.
        \]
        Then $Y\subseteq B\cap C$ and
        \[
            \mu(B\cap C)\leq\mu_{\mathrm{ac}}(B)+\mu_{\mathrm{s}}(C)=0.
        \]
    \end{proof}

    Universal nullity is invariant under Borel isomorphisms.  Fixing
    such an isomorphism between $2^\omega$ and $\omega^\omega$ in $V$,
    Corollary~\ref{cor:preservation} also applies to subsets of
    $\omega^\omega$.  It applies to tails of the iteration over the
    corresponding intermediate models as well.

    \begin{thm}\label{thm:miller_reals_universally_null}
        In the Miller model, the set of Miller reals added at the
        individual stages is universally null.
    \end{thm}
    \begin{proof}
        Let $G$ be $\Miller_{\omega_2}$-generic over
        $V\models\CH$, and put $G_\alpha=G\cap\Miller_\alpha$.
        If $m_\alpha$ is the Miller real added at stage $\alpha$, set
        \[
            X_\alpha=\{m_\xi:\xi<\alpha\}
            \qquad(\alpha\leq\omega_2).
        \]
        We prove by induction on $\alpha$ that $X_\alpha$ is
        universally null in $V[G_\alpha]$.  The case $\alpha=0$ is
        immediate.

        At a successor stage, apply Corollary~\ref{cor:preservation}
        and use $X_{\alpha+1}=X_\alpha\cup\{m_\alpha\}$.

        If $\cf(\alpha)=\omega$, choose a cofinal sequence
        $\seq{\alpha_n:n\in\omega}$ in $\alpha$.  By induction and
        Corollary~\ref{cor:preservation}, each $X_{\alpha_n}$ remains
        universally null in $V[G_\alpha]$, and so does their union
        $X_\alpha$.

        Suppose $\cf(\alpha)>\omega$.  It suffices to consider an
        atomless Borel probability measure $\mu$ on $\omega^\omega$
        in $V[G_\alpha]$.  Countable support iterations of proper
        forcings add no reals at limits of uncountable cofinality.
        Thus the code of $\mu$ belongs to some $V[G_\beta]$ with
        $\beta<\alpha$.  By induction, that model contains a Borel
        set $B\supseteq X_\beta$ with $\mu(B)=0$.

        In $V[G_\beta]$, choose $g\in\omega^\omega$ such that
        \[
            \mu(\{x\in\omega^\omega:x(n)>g(n)\})<2^{-n-1}
            \qquad(n\in\omega).
        \]
        By the Borel--Cantelli lemma, the Borel set
        \[
            H_g=\{x\in\omega^\omega:x\leq^*g\}
            =\bigcup_{N\in\omega}\bigcap_{n\geq N}
                \{x:x(n)\leq g(n)\}
        \]
        has $\mu$-measure one.  The equalities $\mu(B)=0$ and
        $\mu(H_g)=1$ persist in $V[G_\alpha]$.
        For $\beta\leq\xi<\alpha$, Miller genericity over
        $V[G_\xi]$ gives $m_\xi\not\leq^*g$.   Therefore
        \[
            X_\alpha\subseteq B\cup(\omega^\omega\setminus H_g),
            \qquad
            \mu\bigl(B\cup(\omega^\omega\setminus H_g)\bigr)=0.
        \]
        This completes the induction.
    \end{proof}

    \begin{cor}\label{cor:brendle_larson}
        The Miller model contains a nonmeager universally null set of
        cardinality $\aleph_2=\frakc$.
    \end{cor}
    \begin{proof}
        By $\CH$ and \cite[Theorem 0.1]{BrendleLarson2017Nonmeager},
        there is a nonmeager universally null set
        $Y\subseteq\omega^\omega$ in the ground model.
        It remains universally null by Corollary~\ref{cor:preservation}
        and nonmeager by
        \cite[Theorem 10(3)]{Bartoszynski2002PerfectlyMeager}.
        By Theorem~\ref{thm:miller_reals_universally_null},
        $X=\{m_\alpha:\alpha<\omega_2\}$ is universally null and has
        cardinality $\aleph_2$.  Thus $X\cup Y$ is a nonmeager
        universally null set of cardinality $\aleph_2=\frakc$.
    \end{proof}

	\section{Universally null sets in the \texorpdfstring{$\PTfg$}{PTfg} model}
	
	We consider the countable support iteration of $\PTfg$ of length
	$\omega_2$ over a model of $\CH$. Here $f,g$ are admissible
	parameters in
	\cite{BartoszynskiJudahShelah1993Cichon}.
	That paper denotes this forcing by $Q_{f,g}$.
	We use the convention that
	the possible values at coordinate $n$ form the set $f(n)$.
	In particular, the parameters satisfy
	\begin{equation}\label{eq:ptfg_growth}
		f(n)>\prod_{j<n}f(j),
		\qquad
		g(n,k+1)>f(n)^2g(n,k)
		\quad(n,k\in\omega).
	\end{equation}
	We prove that every universally null set in the final extension
	is contained in the reals of some intermediate model, and
	consequently has cardinality at most $\aleph_1$.
	
	We use the terminology of
	\cite[Definition~3.1]{LarsonZapletal2023Polar}.
	An ideal $I$ on a Polish space $Z$, generated by analytic sets,
	is $\boldsig^1_1$-polar if there is an analytic family
	$\mathcal M$ of Borel probability measures on $Z$ such that,
	for every analytic $B\subseteq Z$,
	\[
	B\notin I
	\quad\Longleftrightarrow\quad
	\exists\mu\in\mathcal M\ \mu(B)>0
	\quad\Longleftrightarrow\quad
	\exists\mu\in\mathcal M\ \mu(B)=1.
	\]
	It is iterable if the quotient forcing $P_I$ is proper in
	every forcing extension.
	
	\begin{lem}\label{lem:ptfg_polar}
		The forcing $\PTfg$ has an iterable
		$\boldsig^1_1$-polar ideal presentation.
	\end{lem}
	\begin{proof}
		Put $Z=\prod_n f(n)$, and let $I=I_{\sigma\psi}$ be the
		ideal associated with $\PTfg$ in
		\cite[Section~4.4.3]{Zapletal2008ForcingIdealized}.
		By Proposition~4.4.14 of that reference, for every analytic
		$B\subseteq Z$,
		\begin{equation}\label{eq:ptfg_positive}
			B\notin I
			\quad\Longleftrightarrow\quad
			\exists S\in\PTfg\ ([S]\subseteq B).
		\end{equation}
		In particular, $S\mapsto[S]$ gives a dense embedding into $P_I$.
		
		The summability hypothesis of
		\cite[Proposition~4.4.16]{Zapletal2008ForcingIdealized}
		follows from \eqref{eq:ptfg_growth}: for every $k\in\omega$,
		\[
		\sum_n
		\left(\prod_{j<n}f(j)\right)
		\frac{g(n,k)}{g(n,k+1)}
		\leq
		\sum_n\frac{\prod_{j<n}f(j)}{f(n)^2}
		\leq
		\sum_n\frac1{f(n)}
		<\infty.
		\]
		The last inequality holds since $f(n)>2^n$.
		Here the arguments of $g$ are ordered as in the proof
		on p.~192 of that reference.
		
		For $S\in\PTfg$, let $\mu_S$ be the probability measure
		concentrated on $[S]$ which distributes mass equally among
		the immediate successors of each node of $S$.
		The proof of Proposition~4.4.16 shows that $\mu_S$ vanishes
		on every member of $I$.
		Together with \eqref{eq:ptfg_positive}, this gives, for
		analytic $B\subseteq Z$,
		\[
		B\notin I
		\quad\Longleftrightarrow\quad
		\exists S\in\PTfg\ \mu_S(B)>0
		\quad\Longleftrightarrow\quad
		\exists S\in\PTfg\ \mu_S(B)=1.
		\]
		
		The set of conditions $\PTfg$ is Borel in the space of
		subtrees of $\bigcup_n\prod_{j<n}f(j)$, and the map
		$S\mapsto\mu_S$ is Borel: its values on basic cylinders
		are determined by finitely many successor sets.
		Thus $\{\mu_S:S\in\PTfg\}$ is analytic, so $I$ is
		$\boldsig^1_1$-polar.
		
		These descriptions remain valid in every forcing extension.
		Moreover, properness of $P_I$ follows in every such extension
		from \cite[Theorem~4.4.2]{Zapletal2008ForcingIdealized}.
		Hence $I$ is iterable.
	\end{proof}
	
	We use the following consequence of the results on measured
	extensions in \cite{LarsonZapletal2023Polar}.
	
	\begin{fact}\label{fact:polar_intersection}
		Let $W$ be an extension of $U$ by a countable support
		iteration of an iterable $\boldsig^1_1$-polar forcing.
		Suppose that
		\[
		\seq{B_\mu:\mu\in\PrC^U}\in U
		\]
		is a family of Borel sets such that $\mu(B_\mu)<1$
		for every $\mu\in\PrC^U$.
		Then, interpreting the Borel codes in $W$,
		\[
		\bigcap_{\mu\in\PrC^U}B_\mu^W=\emptyset.
		\]
	\end{fact}
	\begin{proof}
		By \cite[Theorem~5.12]{LarsonZapletal2023Polar},
		$W$ is a measured extension of $U$.
		Apply \cite[Proposition~5.2]{LarsonZapletal2023Polar}.
	\end{proof}
	
	\begin{thm}\label{thm:ptfg_no_large_un}
		Assume $V\models\CH$, and let
		$\seq{\P_\alpha,\dot{\Q}_\alpha:\alpha<\omega_2}$
		be the countable support iteration of $\PTfg$.
		Let $G$ be $\P_{\omega_2}$-generic over $V$, and put
		$V_\alpha=V[G\restrict\alpha]$.
		If $A\subseteq2^\omega$ is universally null in $V_{\omega_2}$,
		then
		\[
		A\subseteq(2^\omega)^{V_\beta}
		\]
		for some $\beta<\omega_2$.
		Consequently, $|A|\leq\aleph_1$, and the $\PTfg$ model
		contains no universally null set of size
		$\frakc=\aleph_2$.
	\end{thm}
	\begin{proof}
		We use the standard properties of this iteration:
		$\P_{\omega_2}$ is $\aleph_2$-cc, every $V_\alpha$ with
		$\alpha<\omega_2$ satisfies $\CH$, and no new real first
		appears at a limit stage of uncountable cofinality.
		
		We use the following standard reflection fact.
		
		\begin{claim}\label{claim:ptfg_reflect_function}
			For every function
			$F:\PrC^{V_{\omega_2}}\to(\omega^\omega)^{V_{\omega_2}}$
			in $V_{\omega_2}$, there is $\beta<\omega_2$ with
			$\cf(\beta)=\omega_1$ such that
			\[
			F\restrict\PrC^{V_\beta}\in V_\beta.
			\]
		\end{claim}
		\begin{proof}
			This follows by a standard reflection argument.
		\end{proof}
		
		Work in $V_{\omega_2}$.
		For each $\mu\in\PrC$, let
		\[
		D_\mu=\{x\in2^\omega:\mu(\{x\})>0\}
		\]
		be its countable set of atoms.
		Choose a $G_\delta$ set $B_\mu$ such that
		\begin{equation}\label{eq:ptfg_null_hulls}
			A\setminus D_\mu\subseteq B_\mu,
			\qquad \mu(B_\mu)=0.
		\end{equation}
		To see that this is possible, let
		$\mu_{\mathrm c}(E)=\mu(E\setminus D_\mu)$ be the
		atomless part of $\mu$.
		Universal nullity gives a $G_\delta$ set $H_\mu\supseteq A$
		with $\mu_{\mathrm c}(H_\mu)=0$.
		Then $B_\mu=H_\mu\setminus D_\mu$ works.
		If $\mu_{\mathrm c}=0$, take $H_\mu=2^\omega$.
		
		Let $F(\mu)$ be a real coding $B_\mu$.
		By Claim~\ref{claim:ptfg_reflect_function}, choose
		$\beta<\omega_2$ such that
		\[
		F\restrict\PrC^{V_\beta}\in V_\beta.
		\]
		In particular, $V_\beta$ contains the entire family of
		codes for $B_\mu$, $\mu\in\PrC^{V_\beta}$.
		By absoluteness of the measure of a coded Borel set,
		$V_\beta$ also satisfies $\mu(B_\mu)=0$ for every
		such $\mu$.
		
		The tail extension $V_{\omega_2}$ of $V_\beta$ is again
		obtained by a countable support iteration of $\PTfg$.
		Lemma~\ref{lem:ptfg_polar} and
		Fact~\ref{fact:polar_intersection} therefore give
		\begin{equation}\label{eq:ptfg_empty_intersection}
			\bigcap_{\mu\in\PrC^{V_\beta}}
			B_\mu^{V_{\omega_2}}=\emptyset.
		\end{equation}
		
		Suppose that $x\in A\setminus V_\beta$.
		If $\mu\in\PrC^{V_\beta}$, all atoms of its
		reinterpretation belong to $V_\beta$: the countable set
		of atoms has an enumeration in $V_\beta$, which remains
		exhaustive in the extension.
		Thus $x\notin D_\mu$, and
		\eqref{eq:ptfg_null_hulls} implies
		$x\in B_\mu^{V_{\omega_2}}$.
		This holds for every $\mu\in\PrC^{V_\beta}$, contradicting
		\eqref{eq:ptfg_empty_intersection}.
		
		Hence $A\subseteq(2^\omega)^{V_\beta}$.
		Since $V_\beta\models\CH$ and the tail forcing preserves
		$\omega_1$, we obtain $|A|\leq\aleph_1$.
		Finally, the full iteration has continuum $\aleph_2$.
	\end{proof}
	
	\section{Forcing destructible universally null sets}\label{sec:destructible}
	
	In his lecture slides \emph{Universally Measurable Sets II},
	dated February~3, 2026, Larson
	\cite[p.~24]{Larson2026UniversallyMeasurableII} leaves open
	whether a universally null set which is not universally meager
	can be made non-universally-null by forcing, asking both whether
	this can happen at all and whether it can be done in general.
	The following theorem answers the existence part affirmatively,
	even with an $\omega_1$-preserving forcing.  The ground model
	set in our example is a Luzin set, hence nonmeager and therefore
	not universally meager.
	We derive the result from the construction of Veli\v{c}kovi\'{c} and
	Woodin in \cite[Section~3]{VelickovicWoodin1998Complexity}.
	The relevant feature of their construction is that the new perfect
	set consists entirely of members of an old family of Cohen reals.
	
	\begin{thm}\label{thm:destructible}
		It is consistent with $\ZFC$ that there are a universally null
		set $A\subseteq2^\omega$ and an $\omega_1$-preserving forcing
		notion $\Q$ such that
		\[
		\Q\forces\text{``$\check A$ contains a nonempty perfect set.''}
		\]
		In particular, $\Q\forces\check A\notin\UN$.
		Moreover, $A$ can be chosen to be a Luzin set in the ground model.
	\end{thm}
	\begin{proof}
		Put $\kappa=\omega_3^L$.
		The construction in
		\cite[Section~3]{VelickovicWoodin1998Complexity}
		gives models $V\subseteq W$, both set-forcing extensions of $L$,
		with the following properties:
		\begin{enumerate}
			\item $V=L[G^*]$, where
			$G^*=\seq{c_\xi:\xi<\kappa}$ is an
			$\Add(\omega,\kappa)$-generic sequence of Cohen reals over $L$;
			\item $\omega_1^V=\omega_1^W$;
			\item in $W$, there is a nonempty perfect set $P$ such that
			\[
			P\subseteq\{c_\xi:\xi<\kappa\}.
			\]
		\end{enumerate}
		Here we use the construction in the proof, rather than only
		the assertion that $W$ contains a perfect set of $V$-reals.
		In that construction, all points of the perfect set in the
		final model are enumerated as $\seq{r_\alpha:\alpha<\omega_1}$
		and assigned to distinct coordinates of $G^*$:
		$c_{\gamma_\alpha}=r_\alpha$.
		This gives the inclusion in (3).
		
		Let $A=\{c_\xi:\xi<\kappa\}\in V$.
		We first show that $A$ is a Luzin set in $V$, that is,
		$A$ is uncountable and meets every meager set in a countable set.
		Work in $V$, and let $B\subseteq2^\omega$ be a meager Borel set.
		Since $\Add(\omega,\kappa)$ is ccc, there is a set
		$I\in L$, countable in $L$, such that a Borel code for $B$
		belongs to $L[G^*\restrict I]$.
		The meagerness of the coded Borel set is absolute between
		these models.
		For every $\xi\in\kappa\setminus I$, the real $c_\xi$ is
		Cohen generic over $L[G^*\restrict I]$, and hence $c_\xi\notin B$.
		Therefore
		\[
		A\cap B\subseteq\{c_\xi:\xi\in I\}
		\]
		is countable.
		Since every meager set is contained in a meager Borel set,
		$A$ is a Luzin set.
		
		For completeness, every Luzin set is universally null.
		Indeed, let $\nu$ be an atomless Borel probability measure
		on $2^\omega$ in $V$.
		A countable dense set is $\nu$-null, so it is contained
		in a $\nu$-null $G_\delta$ set $N$.
		Then $N$ is comeager, and thus $A\setminus N$ is countable.
		Consequently,
		\[
		\nu^*(A)\leq\nu(N)+\nu^*(A\setminus N)=0.
		\]
		Hence $V\models A\in\UN$.
		
		In $W$, however, $P\subseteq A$ is a nonempty perfect set.
		Thus $W\models A\notin\UN$.
		Notice that $A$ here is the original set in $V$.
		
		Finally, since $L\subseteq V\subseteq W$ and $W$ is a
		set-forcing extension of $L$, the intermediate model theorem
		gives a forcing notion $\Q_0\in V$ and a $V$-generic filter
		$H\subseteq\Q_0$ such that $W=V[H]$.
		By the forcing theorem, some condition in $H$ forces both
		that $\check A$ contains a nonempty perfect set and that
		$\omega_1^V$ remains uncountable.
		Restricting $\Q_0$ below this condition gives the required
		$\omega_1$-preserving forcing $\Q$.
	\end{proof}
	
	\begin{rmk}
		In this construction, $V\models\abs{A}=\kappa=\frakc=\aleph_3$,
		whereas $W\models\abs{\kappa}=\aleph_1$.
		Thus the argument uses a collapse of cardinals above $\omega_1$;
		it does not give an example with a ground model satisfying
		$\CH$, or with a cardinal-preserving forcing.
	\end{rmk}
	
	\begin{conjecture}\label{conj:proper-destructible-ch}
		Assume $\CH$.  There are a universally null set
		$A\subseteq2^\omega$ and a proper forcing notion $\Q$ of
		cardinality $\aleph_1$ such that
		\[
		\Q\forces\check A\notin\UN.
		\]
	\end{conjecture}
	
	\section{Cardinal invariants}\label{sec:cardinal_invariants}
	
		\begin{defi}
		Let $\scrI$ be a proper $\sigma$-ideal on the Cantor space containing all singletons.
		\begin{enumerate}
			\item Let $\add(\scrI)$ be the least cardinality of a subfamily $\mathcal{A} \subset \scrI$ such that $\bigcup \mathcal{A} \not \in \scrI$.
			\item Let $\cov(\scrI)$ be the least cardinality of a subfamily $\mathcal{A} \subset \scrI$ such that $\bigcup \mathcal{A} = 2^\omega$.
			\item Let $\non(\scrI)$ be the least cardinality of a subset $A \subset 2^\omega$ such that $A \not \in \scrI$.
			\item Let $\cof(\scrI)$ be the least cardinality of a subfamily $\mathcal{A} \subset \scrI$ which is cofinal: for every $X \in \scrI$, there is $Y \in \mathcal{A}$ such that $X \subset Y$.
			\item For ideals $\scrI \subset \scrJ$, let $\add(\scrI, \scrJ)$ be the least cardinality of a subfamily $\mathcal{A} \subset \scrI$ such that $\bigcup \mathcal{A} \not \in \scrJ$.
		\end{enumerate}
	\end{defi}
	
	Let $\nul$ and $\meager$ denote the Lebesgue null ideal and the meager ideal on $2^\omega$, respectively.
	Also for $\mu \in \MeC$, let $\nul_\mu=\{A\subset 2^\omega:\mu^*(A)=0\}$ denote its null ideal.
	Every nonzero $\mu\in\MeC$ can be normalized to the atomless probability measure $\mu/\mu(2^\omega)$ without changing its null ideal or absolute continuity relations.
	The zero measure satisfies $\nul_0=\Pow(2^\omega)$ and is a least element of $(\MeC,\ll)$, so allowing it does not affect the bounding and dominating numbers below.
	
	\begin{defi}
		Let $(P, \le)$ be a nonempty preordered set with no greatest element.
		Define $\frakd(P, \le)$ to be the least cardinality of a subset $A \subset P$ which is a dominating family in $P$: for every $p \in P$, there is $q \in A$ such that $p \le q$.
		Define $\frakb(P, \le)$ to be the least cardinality of an unbounded family $A \subset P$: for every $p \in P$, there is $q \in A$ such that $\neg (q \le p)$.
	\end{defi}
	
	Let $\le^*$ denote eventual domination on $\omega^\omega$.
	We abbreviate $\frakd(\omega^\omega,\le^*)$ and
	$\frakb(\omega^\omega,\le^*)$ by $\frakd$ and $\frakb$, respectively.
	For an infinite cardinal $\kappa$, let
	$\frakd_\kappa=\frakd(\kappa^\kappa,\le)$, where $\le$ denotes pointwise domination.
	
	An important fact is the following.
	
	\begin{fact}[folklore]
		$\frakb(\MeC, \ll) = \aleph_1$ and
		$\frakd(\MeC, \ll) = \frakc$.
	\end{fact}
	\begin{proof}
		For $\seq{\mu_n}_{n<\omega}\subset\MeC$, put
		\[
		\nu=\sum_{n<\omega}2^{-n-1}
		\frac{\mu_n}{1+\mu_n(2^\omega)}\in\MeC.
		\]
		Then
		\[
		\mu_n \ll \nu \text{ for every } n,
		\]
		which proves $\aleph_1\le\frakb(\MeC,\ll)$.

		Fix a homeomorphism $h:2^\omega\times2^\omega\to2^\omega$
		and let $\Lb$ be the usual product probability measure on $2^\omega$.
		For $x\in2^\omega$, set
		\[
		P_x=h[\{x\}\times2^\omega],\qquad
		\mu_x(B)=\Lb(\{y:h(x,y)\in B\}).
		\]
		Thus $\mu_x\in\MeC$, $\mu_x(P_x)=1$, and
		$P_x\cap P_y=\emptyset$ for $x\ne y$.
		For every $\nu\in\MeC$,
		\[
		S_\nu:=\{x:\mu_x\ll\nu\}
		\subseteq\{x:\nu(P_x)>0\}
		\text{, which implies } |S_\nu|\le\aleph_0.
		\]
		Hence, for $X\subseteq2^\omega$ with $|X|=\aleph_1$,
		\[
		(\forall\nu\in\MeC)(\exists x\in X)
		\ \mu_x\not\ll\nu,
		\qquad \frakb(\MeC,\ll)\le\aleph_1.
		\]
		If $D\subseteq\MeC$ is dominating, then
		\[
		2^\omega=\bigcup_{\nu\in D}S_\nu,
		\qquad \frakc\le |D|\cdot\aleph_0,
		\qquad \frakc\le |D|.
		\]
		Finally,
		\[
		\frakc\le\frakd(\MeC,\ll)\le|\MeC|=\frakc.
		\]
	\end{proof}
	
	\subsection{A remark on previous work}
	
	In the Sacks model, $\UN \subset [2^\omega]^{\le \aleph_1}$
	\cite[Theorem~1.1.4]{Ciesielski_Pawlikowski_2004}.
	If this model also satisfies $2^{\aleph_1}=\aleph_2$, then
	$|[2^\omega]^{\le\aleph_1}|=\aleph_2=\frakc$, so
	$\cof(\UN)\le\frakc$.
	On the other hand, $\aleph_1$ sets of size at most $\aleph_1$
	cannot cover $2^\omega$, so $\cov(\UN)\ge\aleph_2$.
	Consequently, $\cov(\UN)=\cof(\UN)=\frakc$ in this model.
	
	\subsection{Uniformity}
	
	This section identifies the uniformity number of the universally null ideal with the uniformity number of the null ideal.
	
	\begin{prop}\label{prop:uniformity}
		$\non(\UN) = \non(\nul)$.
	\end{prop}
	\begin{proof}
		By $\UN \subset \nul$, we have $\non(\UN) \le \non(\nul)$.
		For every nonzero $\mu \in \MeC$, normalization and the measure isomorphism theorem give $\non(\nul_\mu) = \non(\nul)$.
		For $\mu=0$, every subset of $2^\omega$ belongs to $\nul_\mu$.
		Thus if $\abs{A} < \non(\nul)$, then $A \in \nul_\mu$ for every $\mu \in \MeC$.
	\end{proof}
	
	\subsection{Lower bounds for \texorpdfstring{$\cof(\UN)$}{the cofinality}}
	
	We prove $\frakc\le\cof(\UN)$ by establishing the stronger
	statement that $X\in\UN$ implies $|X|<\cof(\UN)$.
	
	\begin{thm}\label{thm:lower_bound_of_cofinality}
		For every element $X \in \UN$, we have $\abs{X} < \cof(\UN)$.
	\end{thm}
	\begin{proof}
		Fix $X\in\UN$.
		\begin{claim}
			For every function $f:X\to 2^\omega$, the graph
			$\Gamma_f=\{(x,f(x)):x\in X\}$ is universally null in $2^\omega\times 2^\omega$.
		\end{claim}
		\begin{proof}
			Let $\mu$ be an atomless Borel probability measure on $2^\omega\times 2^\omega$, and let
			$\nu(B)=\mu(B\times 2^\omega)$ for Borel $B\subset 2^\omega$.
			Although $\mu$ is atomless, $\nu$ may have atoms.
			The set $D=\{x\in 2^\omega:\nu(\{x\})>0\}$ is countable, and
			$\nu_0(B)=\nu(B\setminus D)$ defines an atomless finite Borel measure on $2^\omega$.
			There is a Borel set $N\supset X$ such that $\nu_0(N)=0$:
			if $\nu_0(2^\omega)>0$, apply $X\in\UN$ to the probability measure $\nu_0/\nu_0(2^\omega)$;
			otherwise, take $N=2^\omega$.
			Now
			\[
			\Gamma_f\subset ((N\setminus D)\times 2^\omega)
			\cup\{(x,f(x)):x\in X\cap D\}.
			\]
			The first set on the right is Borel and has $\mu$-measure
			$\nu(N\setminus D)=\nu_0(N)=0$.
			The second set is countable, hence Borel and $\mu$-null since $\mu$ is atomless.
			Thus $\Gamma_f$ has a Borel $\mu$-null superset.
			As $\mu$ was arbitrary, the claim follows; no measurability of $f$ is needed.
		\end{proof}
		The case $X=\emptyset$ is immediate, so suppose $X\ne\emptyset$ and put $\kappa=\abs{X}$.
		Enumerate $X=\{x_\alpha:\alpha<\kappa\}$ without repetitions.
		Let $\mathcal A\subset\UN$ be any nonempty family of size at most $\kappa$, and
		write $\mathcal A=\{A_\alpha:\alpha<\kappa\}$, allowing repetitions.
		Fix a homeomorphism $h:2^\omega\times 2^\omega\to 2^\omega$.
		For each $\alpha<\kappa$, the set $P_\alpha=h[\{x_\alpha\}\times 2^\omega]$ is a Cantor set
		and therefore supports an atomless Borel probability measure.
		Since $A_\alpha\in\UN$, it cannot contain $P_\alpha$.
		Choose $y_\alpha\in 2^\omega$ such that $h(x_\alpha,y_\alpha)\notin A_\alpha$, and define
		$f:X\to 2^\omega$ by $f(x_\alpha)=y_\alpha$.
		By the claim and the invariance of universal nullity under homeomorphisms,
		$Y=h[\Gamma_f]$ belongs to $\UN$.
		For every $\alpha<\kappa$, however,
		\[
		h(x_\alpha,y_\alpha)\in Y\setminus A_\alpha.
		\]
		Hence $\mathcal A$ is not cofinal in $\UN$.
		Thus no family of size at most $\kappa$ is cofinal, and $\abs{X}<\cof(\UN)$.
	\end{proof}

	\begin{cor}\label{cor:lower_bounds_of_cofinality}
		\begin{enumerate}
			\item $\frakc \le \cof(\UN)$.
			\item $\frakb < \cof(\UN)$.
			\item $\non(\nul) < \cof(\UN)$.
		\end{enumerate}
	\end{cor}
	\begin{proof}
		(1).
		Let $\kappa=\cof(\UN)$ and fix a cofinal family
		$\{A_\alpha:\alpha<\kappa\}\subset\UN$.
		By Theorem \ref{thm:lower_bound_of_cofinality}, $\abs{A_\alpha}<\kappa$ for every $\alpha<\kappa$.
		Since every singleton belongs to $\UN$, cofinality implies
		\[
		2^\omega=\bigcup_{\alpha<\kappa}A_\alpha.
		\]
		Therefore,
		\[
		\frakc=\left|\bigcup_{\alpha<\kappa}A_\alpha\right|
		\le\kappa\cdot\kappa=\kappa=\cof(\UN).
		\]
		
		(2).
		Choose a $\le^*$-increasing unbounded family
		$B=\{f_\alpha:\alpha<\frakb\}\subset\omega^\omega$ such that
		$f_\beta\not\le^* f_\alpha$ whenever $\alpha<\beta<\frakb$.
		The Borel relation $\le^*$ well-orders $B$.
		By Rec\l aw's theorem \cite[Theorem~1]{reclaw}, every set well-ordered by a Borel relation is universally null.
		Thus, identifying $\omega^\omega$ with $2^\omega$ by a Borel isomorphism, we obtain a set
		$X\in\UN$ of size $\frakb$.
		Theorem \ref{thm:lower_bound_of_cofinality} now yields $\frakb<\cof(\UN)$.
		
		(3).
		By Grzegorek's theorem \cite{grzegorek1980solution}, there is a set
		$X\in\UN$ of size $\non(\nul)$.
		Applying Theorem \ref{thm:lower_bound_of_cofinality} to $X$, we obtain
		$\non(\nul)<\cof(\UN)$.
	\end{proof}

	\begin{cor}
		There is no increasing cofinal sequence in $\UN$.
	\end{cor}
	\begin{proof}
		The existence of such a sequence is equivalent to $\add(\UN)=\cof(\UN)$.
		However, Corollary~\ref{cor:lower_bounds_of_cofinality}(3) gives
		$\add(\UN)\le\non(\UN)=\non(\nul)<\cof(\UN)$.
	\end{proof}
	
	\subsection{Generalizing Yorioka's proof}
	
	This section adapts Yorioka's argument for the strong measure zero ideal to the universally null ideal, proving that $\cof(\UN)$ can be computed as a generalized dominating number at $\kappa$ under the assumption $\kappa=\add(\nul)=\frakc$.
	
	\begin{thm}
		If $\add(\nul) = \frakc$, then $\cof(\UN) = \frakd_\frakc$.
	\end{thm}
	\begin{proof}
		This proof is based on Yorioka's proof that CH implies $\cof(\mathcal{SN}) = \frakd_{\aleph_1}$ in \cite{yorioka2002cofinality}.
		
		Let $\kappa := \add(\nul) = \frakc$; in particular, $\kappa$ is regular.
		The usual inequalities between cardinal invariants give
		\[
		\add(\meager)=\cof(\nul)=\non(\nul)=\kappa.
		\]
		Build a cofinal family $\seq{\mu_\alpha : \alpha < \kappa}$ in $(\MeC, \ll)$ and a matrix $\seq{A_\alpha^\beta : \alpha, \beta < \kappa}$ such that
		\begin{enumerate}
			\item $A_\alpha^\beta \subset 2^\omega$ is a dense $G_\delta$ subset such that $\mu_\alpha(A_\alpha^\beta) = 0$.
			\item For each $\alpha<\kappa$, the sequence $\seq{A_\alpha^\beta: \beta < \kappa}$ is increasing and cofinal in $\nul_{\mu_\alpha}$.
			\item For each $\alpha<\kappa$ and each $B\in\nul_{\mu_\alpha}$, we have $\bigcap_{\gamma<\alpha} A_\gamma^0 \setminus B \ne \emptyset$.
		\end{enumerate}
		We claim that such a sequence and a matrix exist.
		Let $\seq{\mu^*_\alpha : \alpha < \kappa}$ be a cofinal family in $(\MeC, \ll)$.
		Assume we have constructed $\seq{\mu_\alpha : \alpha < \alpha'}$ and $\seq{A_\alpha^\beta : \alpha < \alpha', \beta < \kappa}$.
		By $\add(\meager) = \kappa$, the set $\bigcap_{\gamma<\alpha'} A_\gamma^0$ is comeager, so it contains a nonempty perfect subset $P$.
		Choose an atomless Borel probability measure $\nu\in\MeC$ with $\nu(P)=1$, and let $\mu_{\alpha'} := \mu_{\alpha'}^* + \nu$.
		Then $\mu_{\alpha'}$ is a nonzero finite measure and $\mu_{\alpha'}^*\ll\mu_{\alpha'}$, ensuring cofinality of the resulting sequence.
		If $B\in\nul_{\mu_{\alpha'}}$, then $\nu^*(B)=0$, so $P\setminus B\ne\emptyset$; this gives (3).
		Normalization and the measure isomorphism theorem yield $\add(\nul_{\mu_{\alpha'}})=\cof(\nul_{\mu_{\alpha'}})=\kappa$.
		Since $\mu_{\alpha'}$ is atomless and finite, every $\mu_{\alpha'}$-null set is contained in a dense $G_\delta$ $\mu_{\alpha'}$-null set.
		Thus we can recursively choose an increasing cofinal sequence $\seq{A_{\alpha'}^\beta : \beta < \kappa}$ satisfying (1) and (2).
		This finishes the description of the construction.
		For every $F\in\kappa^\kappa$, put
		$A_F=\bigcap_{\alpha<\kappa}A_\alpha^{F(\alpha)}$.
		Cofinality of the family $\seq{\mu_\alpha:\alpha<\kappa}$
		implies $A_F\in\UN$: every $\mu\in\MeC$ is absolutely
		continuous with respect to some $\mu_\alpha$, and
		$A_F\subseteq A_\alpha^{F(\alpha)}$ is $\mu_\alpha$-null.
		
		\begin{claim}
			For every $F \colon \kappa \to \kappa$, there is $G \colon \kappa \to \kappa$ and $\seq{x_\alpha : \alpha< \kappa}$ satisfying:
			\begin{enumerate}[label=(\Alph*)]
				\item $F \le G$ (pointwise domination).
				\item For every $\alpha<\kappa$, we have $\{x_\gamma : \gamma \le \alpha \} \subset A_{\alpha}^{G(\alpha)}$.
				\item For every $\alpha<\kappa$, $x_\alpha \in \bigcap_{\gamma< \alpha} A_\gamma^{G(\gamma)} \setminus A_\alpha^{F(\alpha)}$.
			\end{enumerate}
		\end{claim}
		\begin{proof}
			Assume we have constructed $\seq{G(\gamma), x_\gamma : \gamma < \alpha}$.
			By (3), there is $x_\alpha \in 2^\omega$ such that $x_\alpha \in \bigcap_{\gamma<\alpha} A_\gamma^0 \setminus A_\alpha^{F(\alpha)}$.
			Since $\bigcap_{\gamma<\alpha} A_\gamma^0 \subset \bigcap_{\gamma<\alpha} A_\gamma^{G(\gamma)}$, we have (C).
			Take $G(\alpha) < \kappa$ above $F(\alpha)$ such that $\{x_\gamma : \gamma \le \alpha \} \subset A_\alpha^{G(\alpha)}$. Here, we used $\kappa=\non(\nul)$.
		\end{proof}
		Now assume $\cof(\UN) < \frakd_\kappa$.
		Then there is a basis $\mathcal{B} \subset \UN$ with $\abs{\mathcal{B}} < \frakd_\kappa$.
		For each $B \in \mathcal{B}$, we take $F_B \colon \kappa \to \kappa$ such that $B \subset \bigcap_{\gamma< \kappa} A_\gamma^{F_B(\gamma)}$.
		Since $\abs{\mathcal{B}} < \frakd_\kappa$, we can take $F \colon \kappa \to \kappa$ such that $F \not \le F_B$ for every $B \in \mathcal{B}$.
		By the claim above, we can take $G  \colon \kappa \to \kappa$ and $\seq{x_\alpha : \alpha< \kappa}$ satisfying (A), (B) and (C).
		By (B) and (C), we have $\{x_\alpha : \alpha \in \kappa\} \subset \bigcap_{\gamma<\kappa} A_\gamma^{G(\gamma)}$. Thus $\{x_\alpha : \alpha \in \kappa\} \in \UN$.
		
		Also we now show $\{x_\alpha : \alpha \in \kappa\} \not \subset B$ for every $B \in \mathcal{B}$.
		Fix $B \in \mathcal{B}$. By $F \not \le F_B$, there is $\gamma^* < \kappa$ such that $F_B(\gamma^*) < F(\gamma^*)$.
		Then,
		$B \subset \bigcap_{\gamma< \kappa} A_\gamma^{F_B(\gamma)} \subset A_{\gamma^*}^{F_B(\gamma^*)} \subset A_{\gamma^*}^{F(\gamma^*)}$.
		On the other hand, by (C), we have $x_{\gamma^*} \not \in A_{\gamma^*}^{F(\gamma^*)}$.
		
		This contradicts the assumption that $\mathcal B$ is a basis for $\UN$.
		Thus we have $\frakd_\kappa \le \cof(\UN)$.
		
		For the reverse inequality, take a dominating family
		$D\subseteq\kappa^\kappa$ of size $\frakd_\kappa$.
		Given $B\in\UN$, choose $F_B$ such that
		$B\subseteq A_{F_B}$, and then choose $F\in D$ with
		$F_B\le F$.  The rows of the matrix are increasing, so
		$B\subseteq A_{F_B}\subseteq A_F$.
		Thus $\{A_F:F\in D\}$ is cofinal in $\UN$, proving
		$\cof(\UN)\le\frakd_\kappa$.
	\end{proof}
	
	\subsection{A basic result on additivity}\label{sec:basic_add}
	
	This section studies the additivity of the universally null ideal, relating $\add(\UN)$ to the mixed invariant $\add(\UN,\nul)$.
	
	\begin{rmk}
		Let $\scrI \subset \scrI' \subset \scrJ' \subset \scrJ$ be $\sigma$-ideals.
		Then $\add(\scrI', \scrJ') \le \add(\scrI, \scrJ)$.
	\end{rmk}
	
	\begin{fact}[measure isomorphism theorem]
		Let $A \subset 2^\omega$.
		Then $A\in\UN$ if and only if $f[A]\in\nul$ for every
		Borel isomorphism $f:2^\omega\to2^\omega$.
	\end{fact}
	\begin{proof}
		For the forward implication, pull back the usual product measure
		along $f$.
		Conversely, for any atomless Borel probability measure $\mu$,
		the measure isomorphism theorem gives a Borel isomorphism $f$
		sending $\mu$ to the usual product measure.
		Then $f[A]\in\nul$ implies $\mu^*(A)=0$.
	\end{proof}
	
	\begin{lem}\label{lem:unmapstoun}
		Let $A \in \UN$.
		Then $f[A]\in\UN$ for every Borel isomorphism
		$f:2^\omega\to2^\omega$.
	\end{lem}
	\begin{proof}
		Apply the preceding fact to compositions of Borel isomorphisms.
	\end{proof}
	
	\begin{prop}
		$\add(\UN, \nul) = \add(\UN)$ holds.
	\end{prop}
	\begin{proof}
		The inequality $\add(\UN)\le\add(\UN,\nul)$ follows
		from $\UN\subseteq\nul$.
		For the reverse inequality, fix a family
		$\{A_i:i<\kappa\}\subseteq\UN$ of size
		$\kappa<\add(\UN,\nul)$, and put $A=\bigcup_{i<\kappa}A_i$.
		We show that $A\in\UN$.
		Let $f:2^\omega\to2^\omega$ be a Borel isomorphism.
		Then, by Lemma \ref{lem:unmapstoun}, we have $f(A_i) \in \UN$ for every $i < \kappa$.
		Since $\kappa<\add(\UN,\nul)$, we have $f(A)\in\nul$.
		The preceding characterization now gives $A\in\UN$.
	\end{proof}
	
	\begin{cor}
		$\add(\nul) \le \add(\UN)$.
	\end{cor}
	\begin{proof}
		Since $\UN\subseteq\nul$, we have
		$\add(\nul)\le\add(\UN,\nul)=\add(\UN)$.
	\end{proof}
	
	\subsection{Keeping the additivity small}
	
	Brendle, Cardona, and Mej\'{\i}a \cite{brendle2025separating}
	proved a goodness theorem for Boolean algebras carrying a
	strictly positive finitely additive probability measure and
	used it to keep $\add(\SN)$ small in finite support iterations.
	We refine their argument to obtain an upper bound for
	$\add(\SN,\nul)$ in the same iterations.
	
	Due to results in Section \ref{sec:basic_add} and the well-known inequality $\add(\nul) \le \add(\SN)$, we have the following inequalities.
	
	
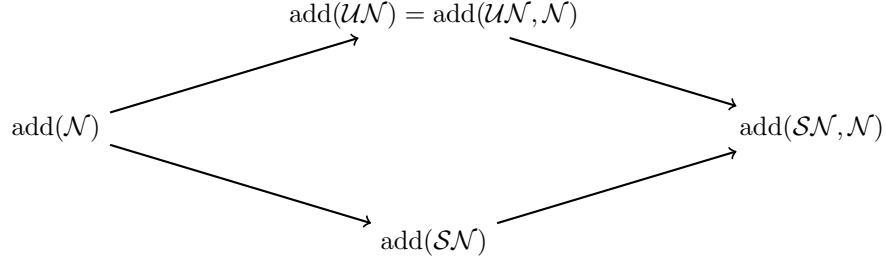
\[
	\begin{tikzpicture}
		\node (addN) at (0, -1.5) {$\add(\nul)$};
		\node (addUN) at (5, 0) {$\add(\UN) = \add(\UN, \nul)$};
		\node (addSN) at (5, -3) {$\add(\SN)$};
		\node (addSN_N) at (10, -1.5) {$\add(\SN, \nul)$};
		
		\draw[thick,->] (addN) to (addUN);
		\draw[thick,->] (addN) to (addSN);
		\draw[thick,->] (addUN) to (addSN_N);
		\draw[thick,->] (addSN) to (addSN_N);
	\end{tikzpicture}
	\]
	
	Thus an upper bound for $\add(\SN,\nul)$ also bounds
	$\add(\UN)$ and $\add(\SN)$.
	We use this observation to prove the consistency of
	$\add(\UN)<\cov(\UN)<\non(\UN)<\cof(\UN)$.
	
	We use Polish relational systems and goodness for them;
	see \cite[Sections~3--4]{brendle2025separating} for the definitions.
	
	\begin{defi}
		Let $\baA$ be a Boolean algebra.
		$\mu \colon \baA \to [0, 1]$ is called a \textit{pfam} (strictly positive finitely additive probability measure) if it satisfies the following:
		\begin{enumerate}
			\item $\mu(1) = 1$,
			\item $\mu(a \vee b) = \mu(a) + \mu(b)$ if $a \wedge b = 0$,
			\item $\mu(a) = 0$ if and only if $a = 0$.
		\end{enumerate}
		We say that the pfam $\mu$ has the \textit{density property} if there is a countable subset $S$ of $\baA \setminus \{ 0 \}$ such that
		$$
		\forall a \in \baA\setminus\{0\}\ \forall \epsilon>0\ \exists s \in S\ \mu(a \wedge s) > \mu(s) (1 - \epsilon).
		$$
	\end{defi}
	
	\begin{defi}
		Define a relational system $\Rpt$ as follows.
		Its domain is $X = 2^{<\omega} \times (2^{<\omega})^\omega$.
		Its codomain is $Y = \{ T : T \subset 2^{<\omega} \text{ is a subtree} \}$.
		For $(t, \sigma) \in X$ and $T \in Y$, let
		$(t, \sigma) \sqsubset_\mathrm{pt} T$ iff $\Lb([T] \cap [t]) > 0$ and $\neg (t \subset \sigma(n) \text{ and } \sigma(n) \in T)$ for all but finitely many $n \in \omega$.
		Here, $\Lb$ denotes the Lebesgue measure on the Cantor space $2^\omega$.
		Put $\Rpt = (X, Y, \sqsubset_{\mathrm{pt}})$.
	\end{defi}
	
	The space $2^{<\omega}$ carries the discrete topology, and $Y$
	is a closed subspace of $2^{2^{<\omega}}$.
	To see that $\Rpt$ is a Polish relational system, write its relation
	as the increasing union, over $k\in\omega$, of the closed relations given by
	\[
	\Lb([T]\cap[t])\ge2^{-k}
	\quad\text{and}\quad
	\forall n\ge k\ \neg(t\subseteq\sigma(n)\in T).
	\]
	For each fixed $T$, these closed sets have nowhere dense sections
	in $X$: when $\Lb([T]\cap[t])>0$, an unrestricted coordinate
	$\sigma(n)$ can be chosen to extend $t$ in $T$.
	
	The following lemma is a refinement of Main Lemma 5.6 in \cite{brendle2025separating}.
	
	\begin{lem}\label{lem:pfam_density}
		Let $\baA$ be a Boolean algebra with a pfam $\mu$ having the density property.
		Then $\baA$ is $\Rpt$-good.
	\end{lem}
	\begin{proof}
		As in the proof of \cite[Main Lemma~5.6]{brendle2025separating},
		we may pass to the completion of $\baA$ and extend $\mu$
		while retaining the density property.
		
		Let $S$ be a witness of the density property.
		Let $\dot{T}$ be a $\baA$-name for a tree.
		Let $b_u = \truth{u \in \dot{T}}$ for $u \in 2^{<\omega}$.
		For $s \in S$ and a rational $\epsilon$ such that $0 < \epsilon < 1$, let
		$$
		T_{s,\epsilon} = \{ u \in 2^{<\omega} : \mu(s \wedge b_u) > \epsilon \mu(s) \}.
		$$
		These are trees in the ground model.
		We show that the countable family $H=\{T_{s,\epsilon}:s\in S,
		\epsilon\in\mathbb Q\cap(0,1)\}$ witnesses goodness.
		Fix $(t,\sigma)$ that is $\Rpt$-unbounded over $H$.
		Suppose, toward a contradiction, that there are
		$b\in\baA\setminus\{0\}$, $m\in\omega$, and
		$q\in\mathbb Q\cap(0,1)$ such that
		\[
		b\forces \Lb([\dot T]\cap[t])\ge q
		\ \text{and}\ \forall n\ge m\
		\neg(t\subseteq\sigma(n)\in\dot T).
		\]
		Take a rational number $\delta > 0$ so small that $\delta< q (1 - \delta)$.
		Take $\epsilon\in\mathbb Q\cap(0,1)$ such that $\delta<\epsilon<q(1-\delta)$.
		By the density property, we can take $s \in S$ such that $\mu(s \setminus b) < \delta \mu(s)$.
		
		For $k\ge|t|$, put $U_k = \{ u \in 2^k : t \subset u \}$.
		Since $b \forces \Lb([\dot{T}] \cap [t]) \ge q$, we have $b \forces \sum_{u \in U_k, u \in \dot{T}} 2^{-k} \ge q$.
		For $E \subset U_k$, let $c_E := b \wedge \bigwedge_{u \in E} b_u \wedge \bigwedge_{u \in U_k \setminus E} \neg b_u$.
		Then $c_E \forces  \sum_{u \in E} 2^{-k} \ge q$.
		So $\sum_{u \in E} 2^{-k} \ge q$ in the ground model provided $c_E \ne 0$.
		Therefore, we have
		\begin{align*}
			\sum_{u \in U_k} 2^{-k} \mu(s\wedge b_u) &\ge 
			\sum_{u \in U_k} 2^{-k} \mu(s \wedge b \wedge b_u) \\
			&= \sum_{E \subset U_k} \sum_{u \in E} 2^{-k} \mu(s \wedge c_E) \\
			&\ge q \sum_{E \subset U_k} \mu(s \wedge c_E) \\
			&= q \mu(s \wedge b) \\
			&> q(1-\delta)\mu(s).
		\end{align*}
		Put $G_k = T_{s,\epsilon} \cap U_k$ and $a_k = \Lb(\bigcup_{u \in G_k} [u]) = 2^{-k} \abs{G_k}$.
		We use the following estimates: if $u \in G_k$, then $\mu(s \wedge b_u) \le \mu(s)$. Also if $u \in U_k \setminus G_k$, then $\mu(s \wedge b_u) \le \epsilon \mu(s)$.
		Using this, we have
		\begin{align*}
			\sum_{u \in U_k} 2^{-k} \mu(s \wedge b_u) &= \sum_{u \in G_k} 2^{-k} \mu(s \wedge b_u) + \sum_{u \in U_k \setminus G_k} 2^{-k} \mu(s \wedge b_u) \\
			&\le a_k \mu(s) + \epsilon \mu(s) \sum_{u \in U_k \setminus G_k} 2^{-k} \\
			&\le (a_k + \epsilon) \mu(s).
		\end{align*}
		Combining the above two inequalities, we have
		$$
		q(1-\delta)\mu(s) < \sum_{u \in U_k} 2^{-k} \mu(s\wedge b_u) \le (a_k + \epsilon) \mu(s).
		$$
		On the other hand, we have $\lim_{k \to \infty} a_k = \Lb([T_{s,\epsilon}] \cap [t])$.
		Thus we have $\Lb([T_{s,\epsilon}] \cap [t]) \ge q (1-\delta) - \epsilon > 0$.
		
		Since $(t,\sigma)$ is $\Rpt$-unbounded over $H$, there is
		$n\ge m$ such that $t\subset\sigma(n)\in T_{s,\epsilon}$.
		Therefore, 
		$\mu(s \wedge b \wedge b_{\sigma(n)}) \ge \mu(s \wedge b_{\sigma(n)}) - \mu(s \setminus b) > (\epsilon - \delta)\mu(s) > 0$.
		Also, $s \wedge b \wedge b_{\sigma(n)} \forces \sigma(n) \in \dot{T}$.
		This is a contradiction to $b \forces (t, \sigma) \sqsubset_{\mathrm{pt}} \dot{T}$ witnessed by $m$.
	\end{proof}
	
	The density property can be omitted.
	\begin{thm}\label{thm:pfam}
		Let $\baA$ be a Boolean algebra with a pfam $\mu$.
		Then $\baA$ is $\Rpt$-good.
	\end{thm}
	\begin{proof}
		We follow \cite[Theorem~5.8]{brendle2025separating}.
		First, the random algebra $\mathbb B_\kappa$ is $\Rpt$-good
		for every infinite cardinal $\kappa$.
		Indeed, a name for a tree depends on countably many random
		coordinates.  On these coordinates, the random algebra has
		a measure with the density property, so
		Lemma~\ref{lem:pfam_density} applies.
		The resulting assertion remains true after forcing with the
		remaining coordinates, by absoluteness of $\sqsubset_{\mathrm{pt}}$.
		Also, every $\sigma$-centered Boolean algebra carries a pfam
		with the density property
		\cite[Lemma~5.5(iii)]{brendle2025separating}, and hence is
		$\Rpt$-good.

		By the embedding theorem used in
		\cite[Theorem~5.8]{brendle2025separating}, $\baA$ completely
		embeds into $\mathbb B_\kappa*\dot Q$ for some
		$\mathbb B_\kappa$-name $\dot Q$ for a $\sigma$-centered poset.
		Goodness is preserved by two-step iterations and passes to
		complete suborders, so $\baA$ is $\Rpt$-good.
	\end{proof}
	
	The following theorem is a refinement of Theorem 5.10 in \cite{brendle2025separating}.
	
	\begin{thm}\label{thm:rpt_iteration}
		Let $\theta_0 \le \theta$ be uncountable regular cardinals,
		let $\lambda=\lambda^{<\theta_0}$ be a cardinal, and put
		$\pi=\lambda\delta$ (ordinal multiplication) for some ordinal
		$0<\delta<\lambda^+$.
		Assume $\theta \le \lambda$ and $\cf(\pi) \ge \theta_0$.
		If $P$ is a finite support iteration of length $\pi$ of
		nontrivial $\theta_0$-cc $\theta$-$\Rpt$-good posets of size
		at most $\lambda$, then $P$ forces
		$(\lambda,[\lambda]^{<\theta},\in)\preceq_\mathrm{T}
		(\SN,\nul,\subset)$.
		In particular, $P$ forces $\add(\SN,\nul)\le\theta$.
	\end{thm}
	\begin{proof}
		Write $\seq{P_\xi,\dot Q_\xi:\xi<\pi}$ for the iteration,
		with $P_\pi=P$, and put $V_\xi=V[G\restrict\xi]$.
		By \cite[Lemma~2.2]{brendle2025separating}, it suffices to
		construct $\{X_\beta:\beta<\lambda\}\subseteq\SN$ such that
		\[
		\bigl|\{\beta<\lambda:X_\beta\subseteq N\}\bigr|<\theta
		\quad\text{for every }N\in\nul
		\]
		in the final extension.
		
		Let $\seq{\gamma_\alpha:\alpha<\pi}$ enumerate $0$ and
		all limit ordinals below $\pi$, and put
		$I_\alpha=[\gamma_\alpha,\gamma_{\alpha+1})$.
		These intervals have order type $\omega$.
		Partition $\pi$ into cofinal sets
		$\seq{Z^0_\beta:\beta<\lambda}$, each of size $\lambda$.
		As in \cite[Theorem~5.10]{brendle2025separating}, bookkeeping
		using $\lambda^{<\theta_0}=\lambda$, the $\theta_0$-chain
		condition, and $\cf(\pi)\ge\theta_0$ gives
		$P_{\gamma_\alpha}$-names $\dot f_\alpha$ such that
		\[
		P_\pi\forces
		\{\dot f_\alpha:\alpha\in Z^0_\beta\}=\omega^{\uparrow\omega}
		\qquad(\beta<\lambda).
		\]
		Here $\omega^{\uparrow\omega}$ is the set of strictly increasing
		functions from $\omega$ to $\omega$.
		The iteration over each interval $I_\alpha$ adds a Cohen real.
		Choose $\sigma_\alpha\in\prod_n2^{f_\alpha(n)}$ in
		$V_{\gamma_{\alpha+1}}$ that is Cohen over $V_{\gamma_\alpha}$,
		and denote its name by $\dot\sigma_\alpha$.
		
		In the final extension, put
		\[
		X_\beta=\bigcap_{\alpha\in Z^0_\beta}[\sigma_\alpha]_\infty,
		\qquad
		[\sigma]_\infty=\{x\in2^\omega:
		\exists^\infty n\ \sigma(n)\subseteq x\}.
		\]
		Every increasing sequence of lengths occurs among the
		$f_\alpha$, $\alpha\in Z^0_\beta$, so $X_\beta\in\SN$.
		
		We first record the unboundedness supplied by these Cohen reals.
		For every $\alpha<\pi$ and $t\in2^{<\omega}$,
		$(t,\sigma_\alpha)$ is $\Rpt$-unbounded over $V_{\gamma_\alpha}$.
		Indeed, if $T\in V_{\gamma_\alpha}$ and
		$\Lb([T]\cap[t])>0$, then for each $m$ it is dense in
		$\prod_n2^{f_\alpha(n)}$ to choose an unrestricted coordinate
		$n\ge m$ with $f_\alpha(n)\ge|t|$ and
		$t\subseteq\sigma_\alpha(n)\in T$.
		This proves the required unboundedness directly in the space
		where $\sigma_\alpha$ is Cohen.
		
		Work in $V$ and fix a $P_\pi$-name $\dot N$ for a null set.
		Choose a nice name $\dot T$ for a nonempty perfect tree such that
		\[
		P_\pi\forces [\dot T]\cap\dot N=\emptyset
		\quad\text{and}\quad
		\forall t\in\dot T\ \Lb([\dot T]\cap[t])>0.
		\]
		Such a tree is obtained by taking a compact set of positive
		measure disjoint from a Borel null hull of $\dot N$, and then
		taking the support of Lebesgue measure restricted to that set.
		
		Let $S=\operatorname{gsupp}_{\Rpt}(\dot T)$ be the goodness
		support from \cite[Definition~4.7]{brendle2025separating}.
		By Lemma~4.8 of that reference, $|S|<\theta$.
		Consequently, the ground model set
		\[
		A=\{\alpha<\pi:S\cap I_\alpha\ne\emptyset\}
		\]
		has size less than $\theta$.
		For $\alpha\notin A$, apply Theorem~4.9 of that reference
		with $\xi=\gamma_\alpha$, $\xi'=\gamma_{\alpha+1}$, and
		$\delta=\pi$.
		It gives a family of trees in $V_{\gamma_\alpha}$ such that
		any point unbounded over that family in $V_{\gamma_{\alpha+1}}$
		remains unbounded by $\dot T$ in the final extension.
		The preceding observation therefore yields
		\[
		P_\pi\forces
		\forall t\in2^{<\omega}\
		(t,\dot\sigma_\alpha)\not\sqsubset_{\mathrm{pt}}\dot T
		\qquad(\alpha\notin A).
		\]
		Put $B=\{\beta<\lambda:Z^0_\beta\cap A\ne\emptyset\}$;
		then $|B|<\theta$.
		
		Fix $\beta\in\lambda\setminus B$.
		Since $\dot T$ is a nice name for a real and
		$\cf(\pi)\ge\theta_0$, it is a $P_{\gamma_\xi}$-name for
		some $\xi\in\pi\setminus Z^0_\beta$.
		Choose a $P_{\gamma_{\xi+1}}$-name $\dot c$ for a Cohen
		real in $[\dot T]$ over $V_{\gamma_\xi}$.
		We show that $P_\pi$ forces
		$c\in X_\beta\cap[T]$, and hence $X_\beta\not\subseteq N$.
		
		Fix $\alpha\in Z^0_\beta$; then $\alpha\ne\xi$.
		If $\xi<\alpha$, then $c\in V_{\gamma_\alpha}$, and Cohen
		genericity of $\sigma_\alpha$ gives
		$c\in[\sigma_\alpha]_\infty$.
		If $\alpha<\xi$, then $\alpha\notin A$, so
		\[
		\forall t\in T\ \exists^\infty n\
		t\subseteq\sigma_\alpha(n)\in T.
		\]
		This assertion is absolute between $V_{\gamma_\xi}$ and the
		final extension.  Thus, in $V_{\gamma_\xi}$, the set
		$\{\sigma_\alpha(n)\in T:n\ge m\}$ is dense in $T$ for
		every $m$.  Cohen genericity of $c$ in $[T]$ again gives
		$c\in[\sigma_\alpha]_\infty$.
		
		It follows that $P_\pi$ forces
		$\{\beta<\lambda:X_\beta\subseteq N\}\subseteq B$.
		Since $|B|<\theta$, the proof is complete.
	\end{proof}
	
	\begin{cor}
		It is consistent that $\add(\UN) < \cov(\UN) < \non(\UN) < \cof(\UN)$.
	\end{cor}
	\begin{proof}
		Start with a model of $\mathsf{GCH}$, and fix regular cardinals
		$\aleph_1<\kappa<\lambda$, for example
		$\kappa=\aleph_2$ and $\lambda=\aleph_3$.
		Apply \cite[Theorem~6.2]{brendle2025separating} to obtain
		\[
		\cov(\nul)=\cov(\SN)=\kappa
		<\cov(\meager)=\frakc=\lambda.
		\]
		The inclusions $\SN\subseteq\UN\subseteq\nul$ give
		$\cov(\UN)=\kappa$, while
		$\cov(\meager)\le\non(\nul)\le\frakc$ gives
		$\non(\UN)=\non(\nul)=\lambda$.
		The forcing is a finite support iteration of length $\lambda$
		whose iterands are subalgebras of random forcing and hence
		carry pfams.
		Theorems~\ref{thm:pfam} and~\ref{thm:rpt_iteration}, with
		$\theta_0=\theta=\aleph_1$, therefore yield
		\[
		\aleph_1\le\add(\UN)\le\add(\SN,\nul)\le\aleph_1.
		\]
		Finally, Corollary~\ref{cor:lower_bounds_of_cofinality}(3)
		gives $\lambda=\non(\nul)<\cof(\UN)$.
	\end{proof}

	\subsection{The diagram}

	The following diagram summarizes the inequalities.
	Every arrow indicates $\le$; the two heavier arrows indicate $<$.
	
	\[
	\begin{tikzpicture}
		\newcommand{\w}{2.6}
		\newcommand{\h}{2.2}
		
		\node (aleph1) at (-\w*0.7, 0) {$\aleph_1$};
		
		\node (addN) at (0, 0) {$\add(\nul)$};
		\node (covN) at (0, \h*2) {$\cov(\nul)$};
		
		\node (addM) at (\w, 0) {$\add(\meager)$};
		\node (b) at (\w, \h) {$\mathfrak{b}$};
		\node (nonM) at (\w, \h*2) {$\non(\meager)$};
		
		\node (covM) at (\w*2, 0) {$\cov(\meager)$};
		\node (d) at (\w*2, \h) {$\mathfrak{d}$};
		\node (cofM) at (\w*2, \h*2) {$\cof(\meager)$};
		
		\node (nonN) at (\w*3, 0) {$\non(\nul)$};
		\node (cofN) at (\w*3, \h*2) {$\cof(\nul)$};
		
		\node (c) at (\w*3.9, \h*3) {$\frakc$};

		\node (addUN) at (\w/2,\h*0.4) {$\add(\UN)$};
		\node (covUN) at (\w/2,\h*2.66) {$\cov(\UN)$};
		\node (cofUN) at (\w*3.9,\h*3.6) {$\cof(\UN)$};
		\node (covSN) at (\w*1.5,\h*3) {$\cov(\mathcal{SN})$};
		\node (cofSN) at (\w*1.5,\h*3.5) {$\cof(\mathcal{SN})$};
		\node[align=center] (nonUN) at (\w*3, 0.3*\h) {$\non(\UN)$ \\ \rotatebox{90}{$=$}};
		
		\draw[thick,->] (aleph1) to (addN);
		\draw[thick,->] (addN) to (covN);
		\draw[thick,->] (addN) to (addM);
		\draw[thick,->] (covN) to (nonM);	
		\draw[thick,->] (addM) to (b);
		\draw[thick,->] (b) to (nonM);
		\draw[thick,->] (addM) to (covM);
		\draw[thick,->] (nonM) to (cofM);
		\draw[thick,->] (covM) to (d);
		\draw[thick,->] (d) to (cofM);
		\draw[thick,->] (b) to (d);
		\draw[thick,->] (covM) to (nonN);
		\draw[thick,->] (cofM) to (cofN);
		\draw[thick,->] (nonUN) to (cofN);
		\draw[thick,->] (cofN) to (c);
		\draw[thick,->] (addN) to (addUN);
		\draw[white, line width=6pt] (addUN) to (covUN);
		\draw[thick,->] (addUN) to (covUN);
		\draw[thick,->] (covN) to (covUN);
		\draw[white, line width=6pt] (addUN) to ($(nonUN)+(-1,0.2)$);
		\draw[thick,->] (addUN) to ($(nonUN)+(-1,0.2)$);
		%\draw[thick,->] (covUN) to (cofUN);
		\draw[thick,->] (covUN) to (covSN);
		\draw[thick,->] (covSN) to (cofSN);
		%\draw[thick,->] (covUN) to (c);
		\draw[thick,->] (covSN) to (c);
		\draw[thick,->] (c) to (cofUN);

		\draw[line width=3pt, ->] (nonUN) to (cofUN);
		\draw[line width=3pt, ->] (b.east) .. controls ($(b)+(6.2,1.2)$) and ($(cofUN)+(-3.5,-0.2)$) .. (cofUN.west);
	\end{tikzpicture}
	\]
	
	\section{Discussion}\label{sec:discussion}
	
	This section lists the remaining questions about the cardinal invariants of $\UN$.
	
	%\begin{question}
	%	Is it consistent that $\cof(\UN) < \frakc$, or does $\ZFC$ prove $\frakc \le \cof(\UN)$?
	%\end{question}
	
	\begin{question}
		In the Laver model or the Mathias model, what is the value of $\cov(\UN)$?
	\end{question}
	
	\begin{question}
		Does $\cof(\SN) \le \cof(\UN)$ hold?
	\end{question}
	
	\begin{question}
		Is it consistent that $\add(\nul) < \add(\UN)$, or does $\ZFC$ prove $\add(\nul) = \add(\UN)$?
	\end{question}
	A natural way to try to increase $\add(\UN)$ is to force $\UN = [2^\omega]^{\le \aleph_1}$. But this statement is inconsistent.
	
	\begin{prop}
		$\UN \ne [2^\omega]^{\le \aleph_1}$.
	\end{prop}
	\begin{proof}
		Suppose $\UN = [2^\omega]^{\le \aleph_1}$.
		Since $2^\omega\notin\UN$, this implies $\frakc>\aleph_1$.
		Hence $\non(\UN)=\non(\nul)=\aleph_2$.
		But, by Grzegorek's theorem \cite{grzegorek1980solution}, there is a universally null set of size $\non(\nul)$.
		This contradicts $\UN = [2^\omega]^{\le \aleph_1}$.
	\end{proof}
	
	\begin{question}
		Is $\add(\UN) < \non(\UN) < \cov(\UN) < \cof(\UN)$ consistent?
	\end{question}

	\section*{Acknowledgments}
	
	This paper is an extended version of the paper presented at the 15th Panhellenic Logic Symposium; the conference version was made available only on the conference webpage and was not published in formal proceedings.
	
	The author thanks Diego Mejía and Tristan van der Vlugt for their helpful comments.

    \section*{Use of AI}

	During the preparation of this work, the author used OpenAI Codex with
	the \texttt{gpt-5.6-sol} and \texttt{gpt-6-astra} models to assist with
	literature searches, the discovery and development of proofs, language
	editing, and \LaTeX{} drafting. The author independently checked,
	reviewed, and revised all AI-assisted material and takes full
	responsibility for the content of the paper.

	\printbibliography
\end{document}